\documentclass{amsproc}
\def\Frac#1#2{\frac{\displaystyle{#1}}{\displaystyle{#2}}}
\def\cali{{\mathcal{I}}}
\newtheorem{theorem}{Theorem}[section]
\newtheorem{lemma}[theorem]{Lemma}
\newtheorem{corollary}[theorem]{Corollary}
\theoremstyle{definition}

\newtheorem{conjecture}[theorem]{Conjecture}

\theoremstyle{remark}
\newtheorem{remark}[theorem]{Remark}

\numberwithin{equation}{section}

\begin{document}

\title{On bounds for ratios of contiguous hypergeometric functions}

\author{Javier Segura}
\address{Departamento de Matem\'aticas, Estad\'{\i}stica y Computaci\'on. Universidad de Cantabria. 39005-Santander.}
\email{javier.segura@unican.es}
\thanks{The author acknowledges support from Ministerio de Ciencia e Innovaci\'on, project
 PID2021-127252NB-I00 with funds from MCIN/AEI/10.13039/501100011033/ FEDER, UE. The author thanks the two anonymous reviewers 
 for many useful comments and suggestions.}


\begin{abstract}
We review recent results on analytical properties (monotonicity and bounds) for ratios of contiguous functions of
hypergeometric type. The cases of parabolic cylinder functions and modified Bessel functions have been discussed with
considerable detail in the literature, and we give a brief account of these results, completing some aspects in the 
case of parabolic cylinder functions.  Different techniques for obtaining these bounds are considered. They
are all based on simple qualitative descriptions of the solutions of associated ODEs (mainly Riccati
equations, but not only Riccati). In spite of their simplicity, they provide the most accurate global bounds known so far. 
We also provide examples of application of these ideas to the more general cases of the Kummer confluent function and the
Gauss hypergeometric function.
The function ratios described in this paper are important functions appearing in a
large number of applications, in which simple approximations are very often required. 
\end{abstract}

\maketitle

\section{Introduction}

Many special functions, and in particular those of hypergeometric type, satisfy first order
differential systems of the form
$$
\begin{array}{l}
y'_n=a_n (x)y_n(x)+d_n(x) y_{n-1}(x),\\
y'_{n-1}=b_n (x) y_{n-1}(x) +e_n (x) y_n (x).
\end{array}
$$
This is the case of the Gauss hypergeometric functions 
$y_n ={}_2{\rm F}_1 (a+\epsilon_1 n,b+\epsilon_2 n;c+\epsilon_3 n;x)$, $n\in {\mathbb N}$, for any
$\epsilon_1,\,\epsilon_2,\,\epsilon_3\in {\mathbb Z}$ and, as a consequence, of the 
confluent hypergeometric (Kummer) function 
$y_n ={}_1{\rm F}_1 (a+\epsilon_1 n;c+\epsilon_3 n;x)$. The functions $y_n(x)$ and 
$y_{n-1}(x)$ are said to be contiguous functions. 

In \cite{Segura:2012:OBF} it is discussed how to obtain bounds for the ratios of contiguous
functions, $h_n(x)=y_n (x)/y_{n-1}(x)$, from the qualitative study of the solutions of the
Riccati equation satisfied by this ratio: 
$$
h'_n (x)=d_n (x)-(b_n(x)-a_n(x))h_n(x)-e_n h_n (x)^2,
$$
or the analogous equation for the reciprocal ratio $y_{n-1} (x)/y_{n}(x)$.
This can be combined with the application of the three-term
recurrence relation
$$
e_{n+1}y_{n+1}(x)+(b_{n+1}(x)-a_n(x))y_n(x)-d_n(x)y_{n-1}(x)=0.
$$

These methods originating from the analysis of the Riccati equation 
have been carried out with particular detail for the case of modified Bessel functions, which
is a sub-case of the confluent hypergeometric family 
with $\epsilon_1=1$, $\epsilon_3=2$ (see \cite[1039.5-6]{Olver:2010:BF}). 
This is an important set of functions with an uncountable number of applications.
It was recently proved in \cite{Segura:2023:SBW} that these methods suffice to characterize 
the best possible bounds of
the form $(\alpha+\sqrt{\beta^2 +x^2})/x$ for the ratios of modified Bessel functions
$I_{\nu-1}(x)/I_{\nu}(x)$ and $K_{\nu+1}(x)/K_{\nu}(x)$. 
An earlier application of these methods to the general confluent case can be found in 
\cite{Segura:2016:SBF}, and later in \cite{Sablica:2022:OBF}.

Modifications of these techniques were considered in  
\cite{Ruiz:2016:ANT} and \cite{Segura:2021:MPF} for modified Bessel functions. In
 the first reference, bounds with improved accuracy are obtained by iterating the process 
 of obtaining bounds from the Riccati equation. In the second, an analysis of the 
 solutions of an equation of the type $\phi'(x)=P(x,f(\phi(x)))$ is considered, with $P(x,y)$ a third degree polynomial
 in $y$, $f$ a simple algebraic function and $\phi(x)$ the double ratio 
 $\phi (x)=h_n (x)/h_{n+1}(x)$.  
 Bounds with improved accuracy are obtained, at the cost of not so simple expressions compared
  to those in \cite{Segura:2023:SBW}. 
 Double ratios were previously considered in \cite{Segura:2021:UVS}
 for Parabolic cylinder functions, and very
 sharp bounds for the ratios of PCFs (sharp in three different limits) were obtained as a
 consequence.
 
 In this paper, we summarize the basic ideas and techniques presented in 
 \cite{Ruiz:2016:ANT,Segura:2021:MPF,Segura:2021:UVS,Segura:2023:SBW}, and we collect the
 most significant results for parabolic cylinder functions and 
 modified Bessel functions. These ideas 
 should be applicable to the more general case of hypergeometric functions, 
 and we apply some of them
 to the Kummer function and the Gauss hypergeometric functions. The analysis of Kummer and Gauss
 functions cases will
 be far from complete; we expect that an exhaustive analysis like the one presented in 
 \cite{Segura:2023:SBW} for modified Bessel functions
 can be carried out, of which the present work can be seen as only the starting point 
 (following the first step in \cite{Segura:2016:SBF} for the Kummer function). In addition, we believe that it should
 be also possible to consider the extended methods of 
 \cite{Ruiz:2016:ANT,Segura:2021:MPF,Segura:2021:UVS} in a more general setting.

\section{Qualitative analysis and bounds}

It appears that most of the results on bounds of ratios of consecutive hypergeometric functions known so far, if not all, 
can be obtained by a qualitative analysis of some first order differential equation satisfied by these or
related ratios, combined with the use of the recurrence relation. 

The most ubiquitous result is probably that concerning the bounds for solutions of Riccati 
equations, which has been used to obtain a good number of sharp inequalities 
for parabolic cylinder functions 
\cite{Segura:2021:UVS}, modified Bessel functions 
\cite{Simpson:1984:SMR,Yuan:2000:OTB,Segura:2011:BFR,Hornik:2013:ABF,Ruiz:2016:ANT,Segura:2021:MPF}
and confluent hypergeometric functions 
\cite{Segura:2016:SBF,Sablica:2022:OBF}. 
Many of these and related results have been also obtained using alternative methods, 
particularly for the case of modified Bessel functions (see for instance 
\cite{Amos:1974:COM,Laforgia:2010:SIF,Baricz:2009:OAP,Yang:2018:MOF}) and also for 
confluent \cite{Kalmykov:2013:LCA,Kalmykov:2013:LCF,Sra:2013:TMW} 
and Gauss hypergeometric functions \cite{Kalmykov:2014:LCO}. 

The approach 
based on the qualitative analysis of first order ODEs, together with the 
application of the recurrence relation, seems to be sufficient for obtaining all the known bounds. This was
explicitly proved in \cite{Segura:2023:SBW} for the case of bounds of the type  
$(\alpha+\sqrt{\beta^2+\gamma^2 x^2})/x$ for the modified Bessel function ratios. In this paper,
we focus on these methods based on the qualitative analysis of ODEs.

Our analysis will start from the construction of bounds from the analysis of Riccati equations, based on the 
following result (see \cite[Theorem 1]{Ruiz:2016:ANT}):

\begin{theorem}
\label{second}
Let $h(x)$ be a solution of $h'(x)=a(x)+b(x)h(x)+c(x)h(x)^2$ with $a(x)c(x)<0$ and $a(x),\,b(x),\, c(x)$ 
continuous in $[a,b]$. Let
 $\lambda (x)$ be the positive solution of $a(x)+b(x)\lambda(x)+c(x)\lambda (x)^2=0$, then 
 the following holds:
 \begin{enumerate}
 \item{If} $c(x)<0$, $h(a^+)>0$, $h' (a^+)\lambda'(a^+)>0$ then $h(x)<\lambda (x)$ if 
 $\lambda'(x)>0$ and $h(x)>\lambda (x)$ if 
 $\lambda'(x)<0$.
 \item{If} $c(x)>0$, $h(b^-)>0$, $h' (b^-)\lambda'(b^-)>0$ then $h(x)<\lambda (x)$ if 
 $\lambda'(x)<0$ and $h(x)>\lambda (x)$ if 
 $\lambda'(x)>0$.
 \end{enumerate} 
 In all these cases $\lambda '(x) h'(x)>0$ for all $x\in (a,b)$, and $h(x)$ is a bounded function 
 in any compact subset of $(a,b)$.
\end{theorem}

\begin{proof}
Let us consider the case $c(x)<0$, $h(a^+)>0$ and $\lambda (x)$ increasing. Then $h'(a^+)>0$, and
 because $c(x)<0$ this is only possible if $h(a^+)<\lambda (a^+)$ because  $h'(x)=a(x)+b(x)h(x)+c(x)h(x)^2$ 
 with $a(x)c(x)<0$. Now, $h(x)$ must be increasing
 unless a value $x_0$ is reached such that $h(x_0)=\lambda (x_0)$, which implies that $h'(x_0)=0$. 
 However this can not occur because 
 the graph of $h(x)$ lies below that of $\lambda (x)$ at the left end of the interval, and $\lambda (x)$ is increasing,
 which means
 that if $h(x_0)=\lambda (x_0)$ then necessarily $h'(x_0)<0$, in contradiction with the fact that $h'(x_0)=0$. 
 Because such $x_0$ does not exist we
 have $0<h(x)<\lambda (x)$ for all $x\in (a,b)$ and then $h'(x)>0$.
 
 The rest of cases follow similarly.
\end{proof}

\begin{remark}
The condition $a(x)c(x)<0$ is not 
 essential, but it simplifies the analysis because the characteristic equation 
 $a(x)+b(x)\lambda(x)+c(x)\lambda(x)^2=0$ has one negative and one positive real solution. For a 
 more general situation, see \cite{Segura:2012:OBF}. On the other hand, the fact that only
 the positive root has been considered is not a restriction, because it is
 always possible to consider the Riccati equation for $-h(x)$ instead of that for $h(x)$.
 \end{remark}

 It is also possible to extract bounds for other type of first order ODEs different from 
 Riccati equations. In particular, in 
 \cite{Segura:2021:MPF,Segura:2021:UVS} an analysis of the 
 solutions of equations $\phi(x)=P(x,f(\phi(x)))$, with $P(x,y)$ a third degree polynomial in $y$,
 $f$ a simple algebraic function and $\phi(x)$ the double ratio $\phi (x)=h_n (x)/h_{n+1}(x)$,
 is considered. The qualitative analysis is similar, but more involved given that we have
 three nullclines to be considered. Later we give more details on these methods when we 
 describe the parabolic cylinder functions and the modified Bessel functions.
 
 As commented before, most of the bounds that are available are related to a nullcline 
 of a first order ODE (as in Theorem \ref{second}). 
 In other cases, the use of an ODE satisfied by these ratios or related functions is also 
 helpful, and one can check if a given function is a 
 bound for one of the solutions by considering the next result. 

\begin{theorem}
\label{prime}
Let $P(x,y)$ be continuous in $(a,b)\times {\mathbb R}$ and $\phi(x)$ be a solution of the ODE 
$\phi'(x)=P(x,\phi(x))$ which is bounded in any compact subset of $(a,b)$. Let $\lambda (x)$ be differentiable in $(a,b)$. Denoting
$\delta (x)=\lambda (x)-\phi(x)$ and $\Delta (x)=\lambda'(x)-P(x,\lambda (x))$ we have that
\begin{enumerate}
\item{}If $\delta (a^+)\Delta(x)>0$ in $(a,b)$ then $\delta (x)\Delta(x)>0$ in $(a,b)$.
\item{}If $\delta (b^-)\Delta(x)<0$ in $(a,b)$ then $\delta (x)\Delta(x)<0$ in $(a,b)$.
\end{enumerate}
\end{theorem}

\begin{proof}
Assume, for instance, that $\delta (a^+)>0$ and $\Delta (x)>0$ in $(a,b)$; the rest of cases follow similarly. 
With this we can prove that no $x_0\in (a,b)$ exists such that $\delta (x)=0$, and because  $\delta (a^+)>0$ this
implies that $\delta (x)>0$ in $(a,b)$. For proving this, we suppose that $x_0$ is the smallest value in $(a,b)$ 
for which $\delta(x_0)=0$  and we arrive at a contradiction, which proves 
that such $x_0$ does not exist. We observe that $\delta (x)>0$ in $(a,x_0)$ because $\phi(x)$ is a bounded
and continuous solution, and $x_0$ is the smallest value on $(a,b)$ 
for which $\delta(x_0)=0$.

Because $\delta'(x)=\lambda'(x)-\phi'(x)=\lambda'(x)-P(x,\phi(x))$, we have 
$\delta'(x_0)=\lambda'(x_0)-P(x_0,\phi(x_0))=\Delta (x_0)$, where the last equality holds because $\delta(x_0)=0$ and then 
$\phi (x_0)=\lambda (x_0)$. But because $\Delta (x_0)>0$, $\delta'(x_0)>0$, which contradicts the fact that
$\delta (x)>0$ in $(a,x_0)$.
\end{proof}

This is a useful result for proving that $\lambda (x)$ is a bound for a given solution of the ODE $\phi'(x)=P(x,\phi(x))$. 
The only information required of the solution $\phi(x)$ is its behavior either at $x=a^+$ or $x=b^-$, and the assumption on
continuity and boundedness. Observe that in Theorem \ref{second} we did not need to require boundedness, and that 
boundedness was a consequence.

All the bounds we will describe in this paper are a consequence of either an analysis of nullclines as in 
Theorem \ref{second} (maybe combined 
with the use of a recurrence relation), some variants for other types
of first order ODEs (as done in \cite{Segura:2021:UVS, Segura:2021:MPF}) or, in some cases where the bounds are 
not directly related to nullclines, as a consequence of Theorem \ref{prime}. 

With respect to the use of recurrence relations, we recall that the families of functions we will consider satisfy 
recurrence relations $y_{n+1}(x)+B_n (x) y_n(x)-C_n(x)y_{n-1}(x)=0$. Denoting $h_n (x)=y_n (x)/y_{n-1}(x)$ and 
applying the recurrence in the backward direction, 
$$
h_n (x)=\Frac{C_n(x)}{B_n(x)+h_{n+1}(x)},
$$
which under certain conditions allows us to obtain a bound for $h_n(x)$ using a bound in the right hand side for $h_{n+1}(x)$. 
For instance, if $C_n (x)$ and $B_n (x)$ are positive and $U_n(x)$ is an upper bound for $h_n(x)$ then $C_n (x)/(B_n (x)+U_{n+1}(x))$
is a lower bound for $h_n(x)$. We refer to \cite{Segura:2012:OBF} for a discussion on the use of forward and backward recursion for
obtaining bounds and its relation to the existence of a minimal solution for the recurrence. 

\section{Parabolic cylinder functions}

We present a brief account on bounds for the ratio 
\begin{equation}
\label{Pn}
\Phi_n (x)=U(n-1,x)/U(n,x), 
\end{equation}
where $U(n,x)$ is the Weber parabolic cylinder function, which is a recessive solution as $x\rightarrow 
+\infty$ of the second order ODE
$$
y''(x)-\left(\Frac{x^2}{4}+n\right)y(x)=0.
$$

\subsection{Sharp bounds from Riccati equations}

In this subsection we summarize the main results of \cite{Segura:2021:UVS}. As an illustration of the ideas
discussed in the previous section, we give some details for the analysis of the Riccati equation. 

The PCF $U(n,x)$ satisfies the following difference-differential system (see 12.8.2 and 12.8.3 of 
\cite{Temme:2010:PCF}):

\begin{equation}
\label{DDEP}
\begin{array}{l}
U'(n,x)=\Frac{x}{2}U(n,x)-U(n-1,x),\\
\\
U'(n-1,x)=-\Frac{x}{2}U(n-1,x)-(n-1/2)U(n,x).
\end{array}
\end{equation}

This system is the only information required to obtain the bounds, together with the fact that the function $\Phi_n(x)$
 is positive for all real $x$ when $n>1/2$ 
 and increasing for large $x$ , which is easy to check from the asymptotic expansions of $U(n,x)$
\cite[12.9.1]{Temme:2010:PCF}. Using that expansion  we see that as $x\rightarrow +\infty$ \footnote{We note an obvious 
erratum in \cite[Eq. (8)]{Segura:2021:UVS}: the exponent of the third term should be $-4$ and not $4$}, 
\begin{equation}
\label{u+}
\Phi_n (x)\sim x\left[1+(n+1/2)x^{-2}
-(n+1/2)(n+3/2)x^{-4}+{\mathcal O}(x^{-6})
\right],
\end{equation}
and the first term is enough to see that, indeed, $\Phi_n(x)$ is positive and increasing as $x\rightarrow +\infty$.

On the other hand,
as $x\rightarrow -\infty$ \footnote{Again, the exponent of the third term in \cite[Eq. (9)]{Segura:2021:UVS} 
should be $-4$ and not $4$},
\begin{equation}
\label{u-}
\Phi_n (x)\sim -\Frac{n-1/2}{x}\left[1-(n-3/2)x^{-2}
+2(n-3/2)(n-2)x^{-4}+{\mathcal O}(x^{-6})
\right].
\end{equation}
(notice that $\Phi_n (x)$ becomes negative for $n< 1/2$ and $x\rightarrow -\infty$).  

Using (\ref{DDEP}) we see that $\Phi_n (x)$ is one of the solutions of the Riccati equation
\begin{equation}
\label{ricpcf}
h'(x)=h(x)^2-x h (x)-(n-1/2),
\end{equation}
which for $n>1/2$ has the positive and increasing characteristic root $\lambda(x)=\frac12\left(x+\sqrt{x^2+4n-2}\right)$.
And because of (\ref{u+}) the hypothesis of Theorem \ref{second} are met, with $(a,b)=(-\infty,+\infty)$, $c(x)>0$ and $\lambda'(x)>0$,
which implies that $\Phi_n(x)>\lambda(x)$ for all real $x$. In addition, applying the recurrence relation in the backward and the
forward directions two additional (upper) bounds are obtained for $\Phi_n(x)$. We summarize those results in the following theorem. 
The second bound in the next theorem is obtained from the first bound and by applying the backward
recurrence
\begin{equation}
\label{back}
\Phi_n (x)=x+\Frac{n+\frac12}{\Phi_{n+1}(x)},
\end{equation}
and the third bound is a consequence of the forward recurrence
\begin{equation}
\label{for}
\Phi_n (x)=\Frac{n-\frac12}{-x+\Phi_{n-1}(x)}.
\end{equation}

\begin{theorem}
\label{oldbp}
The following bounds hold for all real $x$
$$
\Frac{U(n-1,x)}{U(n,x)}>B^{(2,1)}(x)=\Frac{1}{2}\left(x+\sqrt{x^2+4n-2}\right) \mbox{ for } n>1/2,
$$
$$
\Frac{U(n-1,x)}{U(n,x)}<B^{(1,2)}(x)=\Frac{1}{2}\left(x+\sqrt{x^2+4n+2}\right) \mbox{ for } n>-1/2,
$$
$$
\Frac{U(n-1,x)}{U(n,x)}<B^{(3,0)}(x)=\Frac{1}{2}\Frac{n-1/2}{n-3/2}\left(x+\sqrt{x^2+4n-6}\right) \mbox{ for } n>3/2.
$$
\end{theorem}

The notation $B^{(m,n)}(x)$ for the bounds is analogous to that used for modified Bessel functions in 
\cite{Segura:2023:SBW}, and denotes the number of terms of the expansions of 
\begin{equation}
\label{B1}
B(\alpha,\beta,\gamma,x)=\alpha x+
\sqrt{\beta^2 x^2 +\gamma^2},\, \alpha>0,
\end{equation}
 which coincide with those of (\ref{u-}) 
as 
$x\rightarrow -\infty$ ($m$) 
and (\ref{u+}) as $x\rightarrow +\infty$ ($n$).  

Assuming $\beta>0$ we have
$$
B(\alpha,\beta,\gamma,x)=(\alpha\pm \beta)x\pm \Frac{\gamma^2}{2\beta x}\mp\Frac{\gamma^4}{8\beta^3 x^3}
\pm \Frac{\gamma^6}{16\beta^5 x^5}+{\mathcal O}(x^{-7}) \mbox{ as } x\rightarrow \pm \infty.
$$

Since in all the three bounds in the previous theorem $\alpha=\beta$, the first term in the previous expansion 
as $x\rightarrow -\infty$ is zero, coinciding with the expansion (\ref{u-}); therefore $m\ge 1$ for all of them. Also, we
can easily check that the first bound has an additional correct term as $x\rightarrow -\infty$ and another one as
$x\rightarrow +\infty$. Similarly, we  can check the accuracy for the other two bounds. We observe that with three parameters,
we have $m+n\le 3$. It appears that the bound with $m=0$, $n=3$ should be possible.

Indeed, it is possible to give a bound of the type (\ref{B1}) with higher accuracy as $x\rightarrow +\infty$, by choosing
$\alpha$, $\beta$ and $\gamma$ such that the first three terms in the expansions as $x\rightarrow +\infty$ of  $\Phi_n (x)$ and  $B(\alpha,\beta,\gamma,x)$ coincide. This gives the values $\alpha=\frac12 (n+5/2)/(n+3/2)$, $\beta=\frac12 (n+1/2)/(n+3/2)$ and $\gamma^2 =(n+1/2)^2/(n+3/2)$, and
one can prove that this is indeed a bound using Theorem \ref{prime}, though unsharp as $x\rightarrow -\infty$. 

\begin{theorem}
\label{newbp}
If $n>-1/2$ then for all real $x$
$$
\Frac{U(n-1,x)}{U(n,x)}>B^{(0,3)}(x)=\Frac{(n+5/2)x+(n+1/2)\sqrt{x^2+4n+6}}{2(n+3/2)}.
$$
\end{theorem}
\begin{proof}
We start by checking that, denoting 
$$\lambda_n (x)=\Frac{(n+5/2)x+(n+1/2)\sqrt{x^2+4n+6}}{2(n+3/2)},$$
 we have
$$
\lambda_n (x)-\Phi_n (x)=-\Frac{(n+1/2)(n+3/2)}{x^5}+{\mathcal O}(x^{-7}),\,x\rightarrow +\infty
$$

With this, and considering Theorem \ref{prime}, we only have to prove that $\Delta (x)>0$ for real $x$, where
$\Delta (x)=\lambda' (x)-P(x,\lambda (x))$, with $P(x,y)=y^2-xy-(n-1/2)$
(see (\ref{ricpcf})).

We have 
$$
\Delta (x)=\Frac{n+1/2}{(2n+3)^2}\left(2x^2+2n+3-\Frac{x(2x^2+6n+9)}{\sqrt{x^2+4n+6}}\right),
$$
and we see that $\Delta(x)>0$ if $n>-1/2$ and $x\le 0$.
After some elementary algebra we can also write
$$
\Delta (x)=\Frac{2(n+1/2)(2n+3)}{\sqrt{x^2+4n+6}\left((2 x^2+2n+2)\sqrt{x^2+4n+6}+x(2x^2+6n+9))\right)}
$$
and $\Delta (x)>0$ if $n>-1/2$ and $x>0$, which completes the proof.
\end{proof}

It is also possible to prove the previous bound in a more straightforward way 
starting from the second bound in Theorem \ref{oldbp} and applying the backward recurrence
(\ref{back}). We give the proof using Theorem \ref{prime} as an illustration of application
of this theorem.

The bound in Theorem \ref{newbp} is new, unlike the bounds in Theorem \ref{oldbp}, which were already discussed in 
\cite{Segura:2021:UVS}. We observe that the bound in Theorem \ref{newbp} becomes negative for $x<-(n+1/2)$, which is a clear indication
of the unsharpness as $x\rightarrow -\infty$. In any case, it is the best bound of the form (\ref{B1}) as
$x\rightarrow +\infty$, and it completes the set of best bounds in the same way that the set of best bounds for ratios 
of modified Bessel functions was completed in \cite{Segura:2023:SBW}, as we will also describe later in this paper. 

For the case of modified Bessel functions, 
uniparametric sets of bounds linking the bounds of type $B^{(0,3)}$ with those of type $B^{(2,1)}$ and the bounds $B^{(3,0)}$ with the $B^{(1,2)}$ bounds were given in \cite{Segura:2023:SBW}, as we will later summarize in Theorem \ref{horn}. It is an open question whether the same type of analysis is possible for parabolic cylinder functions. 
Similarly, it seems feasible that best bounds could be found with have an 
osculatory character, meaning that the graphs of the bounds and the ratio of parabolic cylinder functions would be tangent at a point. 

Subsequent applications of the recurrence are 
possible to obtain further bounds, but the resulting bounds become more complicated and are 
no longer of the form (\ref{B1}). For instance, considering the last bound
in Theorem \ref{oldbp} and applying a further step of forward recurrence (\ref{for}) we get
\begin{theorem}
Let $n>5/2$, the following bound holds for real $x$
$$
\Frac{U(n-1,x)}{U(n,x)}>B^{(4,0)}(x)=\Frac{2(n-1/2)(n-5/2)}{(n-3/2)\sqrt{x^2+4n-10}-(n-7/2)x}.
$$
\end{theorem}
This bound has the same three first terms as the expansion (\ref{u-}), but it is unsharp
as $x\rightarrow +\infty$.

We notice that all the bounds given so far for parabolic cylinder functions are a consequence
of the first bound in Theorem \ref{oldbp}, which was obtained from the Riccati equation, and the
application of the recurrence relation. The backward recurrence improves
the accuracy at $+\infty$ but worsens it at $-\infty$; the opposite occurs with 
the forward recurrence. The only bound which is sharp at $\pm \infty$ is the 
first bound in Theorem \ref{oldbp}. 
For obtaining bounds with higher accuracy both as $x\rightarrow -\infty$ and $x\rightarrow +\infty$,
a different approach should be considered.

\subsection{Beyond the Riccati bounds}
\label{beric}

Riccati equations are not the only possibility of obtaining sharp bounds for ratios of parabolic cylinder functions
and other functions of hypergeometric type. As we see next, it is possible to use other differential equations which  
can give even sharper bounds. The possibility considered in \cite{Segura:2021:UVS} is to analyze the first order 
differential equation satisfied by the double ratio $\Phi_n(x)/\Phi_{n+1}(x)$. The rationale behind this approach is the fact
that the double ratio has a slower variation than the simple ratio, and this facilitates finding more accurate bounds. Similar
ideas were also considered for modified Bessel functions, as we later describe.

In \cite{Segura:2021:UVS} it was shown that the function
$$
W_n (x)=\left(n+\frac12\right)\Frac{\Phi_n (x)}{\Phi_{n+1}(x)}
$$
satisfies the differential equation
\begin{equation}
\label{ri2}
W_n' (x)=2\left(\phi_n(x)^2-V_n(x)\phi_n (x)-\Frac{x}{4}\right),\,V_n(x)=\Frac{x^2}{4}+n
\end{equation}
where 
$$
\phi_n (x)=\Phi_n (x)-\Frac{x}{2}=\displaystyle\sqrt{\Frac{x^2}{4}+W_n(x)}.
$$

An analysis similar to that carried out for Riccati equations in (\ref{second}) can be considered 
for this differential equation. The analysis is more involved because solutions of a
third degree equation have to be considered and we need to prove the monotonicity of some 
functions related to the roots, 
particularly for the largest root \cite[Lemma 5]{Segura:2021:UVS}. For details we refer to \cite{Segura:2021:UVS}; we just recall 
Lemma 6 of that reference, and its consequence for the bounds.

\begin{lemma}
\label{lemmaprin}
Let $y(x)$ satisfy the differential equation
\begin{equation}
\label{ypz}
y'(x)= 2\left(z(x)^3-\left(\Frac{x^2}{4}+n\right)z(x)-\Frac{x}{4}\right),\,n>1/2
\end{equation}
where
\begin{equation}
\label{yz}
z(x)=+\sqrt{\frac{x^2}{4}+y(x)}.
\end{equation}
If $y(x)$ is positive and increasing as $x\rightarrow +\infty$, then 
$$
z(x)>\lambda_n^+(x)=f_n(x) \cos\left(
\frac13\arccos\left(\Frac{x}{f_n(x)^3}\right)\right),\,f_n(x)=\sqrt{\Frac{x^2+4n}{3}}
,
$$
$$
y(x)>\lambda_n^+(x)^2-\Frac{x^2}{4}
$$
and $y'(x)>0$ for all real $x$.
\end{lemma}

From the expansion (\ref{u+}) it is easy to check that $W_n(x)$ satisfies the hypothesis for
$y(x)$ in the previous theorem, from which 
lower bounds for $W_n (x)$ and $\Phi_n (x)$ are derived.
In addition, the monotonicity of $W_n (x)$ is also proved. 

The following result gives the trigonometric bound for $\Phi_n(x)$ that stems from the previous theorem plus
an additional algebraic bound with similar accuracy that is obtained using similar ideas as in Theorem \ref{prime}.

\begin{theorem}
\label{trip}
The following holds for any real $x$ and $n>1/2$
$$
\Frac{U(n-1,x)}{U(n,x)}-\Frac{x}{2}>f_n(x) \cos\left(
\frac13\arccos\left(\Frac{x}{f_n(x)^3}\right)\right)>\displaystyle\sqrt{\Frac{x^2}{4}+g_n (x)},
$$
where $f_n(x)=\sqrt{\Frac{x^2+4n}{3}}$, $g_n(x)=\left(n+\frac12\right)\Frac{x+\sqrt{x^2+4n-2}}{x+\sqrt{x^2+4n+2}}$.
\end{theorem}

Both the trigonometric and the algebraic bounds are very sharp as $x\rightarrow \pm \infty$; the first three terms in the expansion 
(\ref{u+}) are reproduced, and the two first terms in (\ref{u-}). With the notation $B^{(n,m)}$ used before, these are $B^{(3,3)}$ bounds (recall that we considered $n\ge 1$ if the bound is ${\mathcal O}(x^{-1})$ 
as $x\rightarrow -\infty$, and then two correct terms means $n=3$). They are also very sharp as $n\rightarrow +\infty$, see
\cite{Segura:2021:UVS}.

The forward and backward recurrences can again be considered. Starting with 
the bounds in Theorem \ref{trip} we get the upper bound $B^{(2,4)}$ by using the backward recurrence, while
the forward recurrence gives the bound $B^{(4,2)}$.

The monotonicity of the double ratio $W_n (x)$ was
earlier proved in \cite{Koch:2020:UBA} with a more indirect probabilistic approach. 
The original motivation of 
\cite{Segura:2021:UVS} was to prove that property by a direct method, but very sharp bounds were also obtained
as a consequence.
We end this section formulating a conjecture that generalizes the property of 
monotonicity of the double ratio $W_n(x)$:

\begin{conjecture}
Let $R_n^{[1]} (x)=U(n-1,x)/U(n,x)$, $n>1/2$, and define $R_n^{[k+1]}=R_{n}^{[k]}(x)/ R_{n+1}^{[k]}(x)$, then
 the functions $R_{n}^{[k]}(x)$ are positive increasing functions of $x$ with
$R_n^{[k+1]}(x)>R_n^{[k]}(x)$. $R_{n}^{[k]}(x)<1$ if $k\ge 2$.
\end{conjecture}

The ratios $R_{n}^{[k]}(x)<1$  have a sigmoidal shape for $k\ge 2$. From (\ref{u+}) we
see that $\displaystyle\lim_{x\rightarrow +\infty}R_{n}^{[k]}(x)=1$ and numerical experiments
seem to suggest that $\displaystyle\lim_{k\rightarrow +\infty}R_{n}^{[k]}(x)=1$ for $n>1/2$, 
but not close to $n=1/2$ (observe that (\ref{u-}) indicates that $R_{1/2}^{[k]}(-\infty)=0$).

\section{Modified Bessel functions}

Modified Bessel functions are, without any doubt, the functions of hypergeometric type 
for which the analysis of the bounds and monotonicity properties for the ratios of these
functions have been more deeply studied (see \cite{Amos:1974:COM,Simpson:1984:SMR,
Yuan:2000:OTB,Baricz:2009:OAP,
Laforgia:2010:SIF,Segura:2011:BFR,Hornik:2013:ABF,Ruiz:2016:ANT,Yang:2018:MOF,
Segura:2021:MPF,
Segura:2023:SBW}). This is not surprising, given the huge amount of applications where these
ratios appear (see, for example, the applications cited in 
\cite{Segura:2011:BFR,Segura:2023:SBW}). In most of the papers (with the exception of \cite{Ruiz:2016:ANT,Segura:2021:MPF}) 
the bounds are of the form 
\begin{equation}
\label{type}
B(\alpha,\beta,\gamma,x)=\Frac{\alpha+\sqrt{\beta^2+\gamma^2 x^2}}{x}.
\end{equation}
These bounds 
are widely used because they can be quite sharp, they are simple and it is easy to operate
with them. In \cite{Segura:2023:SBW}, the analysis of these type of bounds was concluded,  and the best possible
bounds of this form were characterized and classified.

As for the rest of cases discussed in this paper, the main piece of information in our analysis is 
the difference-differential system 
\cite[10.29.2]{Olver:2010:BF}
\begin{equation}
\label{DDE}
\begin{array}{l}
\cali'_{\nu}(x)=\cali_{\nu-1}(x)-\Frac{\nu}{x}\cali_{\nu}(x),\\
\\
\cali'_{\nu-1}(x)=\cali_{\nu}(x)+\Frac{\nu -1}{x}\cali_{\nu-1}(x)
\end{array}
\end{equation}
(where $\cali_{\nu}(x)$ denotes $I_{\nu}(x)$, $e^{i \pi \nu} K_{\nu}(x)$ or any linear
 combination of them),  together 
with the unique behavior of  $I_{\nu}(x)$ as $x\rightarrow 0^+$ and of
$K_{\nu}(x)$ as $x\rightarrow +\infty$. 

In the next section, we briefly summarize the main results given in \cite{Segura:2023:SBW}. The techniques employed are 
similar to the ideas of Theorems \ref{second} and \ref{prime}. After this, we summarize other types of bounds with
higher accuracy (but not so simple), both of algebraic and trigonometric type.

\subsection{Best bounds of the type 
{\boldmath $B(\alpha,\beta,\gamma,x)=(\alpha+\sqrt{\beta^2+\gamma^2 x^2})/x$}}

One of the main results proved in \cite{Segura:2023:SBW} is that if $\alpha$, 
$\beta$ and $\gamma$ are chosen
such that $B(\alpha,\beta,\gamma,x)=(\alpha+\sqrt{\beta^2+\gamma^2 x^2})/x$ is a sharp approximation for $\Phi_{\nu}(x)=I_{\nu-1} (x)/I_{\nu}(x)$ 
as $x\rightarrow 0^+$
 (respectively $x\rightarrow +\infty$) and the graphs of the functions $B(\alpha,\beta,\gamma,x)$ and 
$\Phi_{\nu}(x)$ are tangent at some $x=x_*>0$, then $B(\alpha,\beta,\gamma,x)$ is an upper (respectively lower) 
bound for $\Phi_{\nu}(x)$; the same is true for the ratio $\Phi_{\nu}(x)=K_{\nu+1} (x)/K_{\nu}(x)$ 
but interchanging lower and upper bounds. This provides the best possible bounds of the form 
$B(\alpha,\beta,\gamma,x)$ around any chosen value $x_*$. 

There is no explicit expression for all 
the coefficients $\alpha$, $\beta$ and $\gamma$ for the best bounds, except in the limits $x_*\rightarrow +\infty$ and 
$x_*\rightarrow 0$, when they give the best possible bounds at $x=0$ and/or $x=+\infty$.
The best possible bounds at $x=0$ and/or $x=+\infty$ are therefore explicitly known, as
 described in \cite{Segura:2023:SBW}.

These best bounds at $x=0$ and/or $x=+\infty$ are particular cases of the four parametric 
bounds given in \cite{Segura:2023:SBW} for $I_{\nu-1}(x)/I_\nu (x)$ and $K_{\nu+1}(x)/K_{\nu}(x)$. These are the most
accurate known bounds of the type $B(\alpha,\beta,\gamma,x)$ with explicit formulas; they are close to the best bounds 
described above, and they contain as particular cases the best bounds at $x=0,+\infty$. We next summarize
these four parametric bounds (upper and lower bounds both for $I_{\nu-1}(x)/I_\nu (x)$ and $K_{\nu+1}(x)/K_{\nu}(x)$).
The first of these four theorems was already given in \cite{Hornik:2013:ABF} in a different form; 
the other three are given in \cite{Segura:2023:SBW}.

\begin{theorem} 
\label{horn}
The following holds for $\lambda\in [0,1/2]$, $\nu\ge \frac12 -\lambda$ and $x>0$:
\begin{equation}
\begin{array}{r}
\label{tang}
\Frac{I_{\nu-1}(x)}{I_{\nu}(x)}> L_\nu^{(I)}(\lambda,x)=B(\alpha^{(I)}_{\nu}(\lambda),\beta^{(I)}_{\nu}(\lambda),1,x),
\end{array}
\end{equation}
where $\alpha_{\nu}^{(I)}(\lambda)=\nu-1/2-\lambda$, $\beta_{\nu}^{(I)} (\lambda)=\sqrt{2\lambda}+\sqrt{\nu^2-(\lambda-\frac12)^2}$.
\end{theorem}
\begin{theorem} 
\label{bnk}
The following holds for $\lambda\in[0,1/2]$, $\nu\ge \frac12-\lambda$ and $x>0$:
$$
\Frac{K_{\nu+1}(x)}{K_{\nu}(x)}<U_{\nu}^{(K)}(\lambda,x)=B(\alpha_\nu^{(K)}(\lambda),\beta_\nu^{(K)}(\lambda),1,x),
$$
where
$$
\alpha_\nu^{(K)}(\lambda)=\nu+1/2+\lambda,\,\beta_\nu^{(K)}(\lambda)=-\sqrt{2\lambda}+\sqrt{\nu^2-(\lambda -\frac12)^2}.
$$

\end{theorem}
\begin{theorem} 
\label{IU}
The following holds for $\nu\ge 0$, $\lambda\in [1/2,2]$ and $x>0$:
\begin{equation}
\label{iuq}
\Frac{I_{\nu-1}(x)}{I_{\nu}(x)}<U_\nu^{(I)}(\lambda,x)=B(\nu-\lambda,\nu+\lambda,\sqrt{c_\nu^{(I)}(\lambda)},\,x),
\end{equation}
where
$$
c_\nu^{(I)}(\lambda)=\Frac{\nu+\lambda}{\nu-\lambda+2\sqrt{2\lambda}-1}.
$$
\end{theorem}
\begin{theorem} 
\label{blk}
The following holds for  $\lambda\in [1/2,2]$, $\nu\ge \lambda$ and $x>0$:
$$
\Frac{K_{\nu+1}(x)}{K_{\nu}(x)}>L_{\nu}^{(K)}(\lambda,x)=B(\nu+\lambda,\nu-\lambda, \sqrt{c_{\nu}^{(K)}(\lambda)}, x),
$$
where
$$
c_\nu^{(K)}(\lambda)=\Frac{\nu-\lambda}{\nu+\lambda-2\sqrt{2\lambda}+1}.
$$
\end{theorem}

As we did earlier in this paper, it is possible to find the bounds of the form $B(\alpha,\beta,\gamma,x)$ which are most accurate
at $x=0$ or $x=+\infty$ by comparing the expansions at these points. We say that a bound has accuracy $(n,m)$ if the first $n$
 terms of its expansion around $x=0$ are exact and the same happens with the first $m$ terms at $x=+\infty$. In the next table
 we summarize such bounds both for $I_{\nu-1}/I_{\nu}(x)$ and $K_{\nu+1}/K_{\nu}(x)$; all those bounds are particular cases of the 
 previous four theorems, and this relation is also given in the table.
  
  For the ratio $I_{\nu-1}(x)/I_{\nu}(x)$, all the bounds appearing in the table had been already described before, but 
  they were first classified and ordered in \cite{Segura:2023:SBW}. 
  In contrast, the set of best bounds at $x=0$ and/or $x=+\infty$ for 
  $K_{\nu+1}(x)/K_{\nu}(x)$ was not complete, as the cases $(0,3)$ and $(3,0)$ had not been considered earlier and were first described in 
  \cite{Segura:2023:SBW}.
  
 \begin{table}[h]
\label{table1}
\begin{tabular}{ccccll}
$(n,m)$ & $\alpha$ & $\beta$ & $\gamma$ & Range & Bound\\
$(2,1)$ & $\nu -1$ & $\nu+1$ & $1$ & $\nu\ge 0$ &  $L_{\nu}^{(I)}(\frac12,x)$\\
$(0,3)$ & $\nu -\frac12$ & $\sqrt{\nu^2-\frac14}$ & $1$ & $\nu\ge \frac12$ & $L_{\nu}^{(I)}(0,x)$ \\
$(1,2)$ & $\nu -\frac12$ & $\nu+\frac12$ & $1$ & $\nu\ge 0$ &  $U_{\nu}^{(I)}(\frac12,x)$ \\
$(3,0)$ & $\nu -2$ & $\nu+2$ & $\sqrt{(\nu+2)/(\nu+1)}$ & $\nu \ge 0$ & $U_{\nu}^{(I)}(2,x)$ \\
$(2,1)$ & $\nu+1$ & $\nu-1$ & $1$ & $\nu\in {\mathbb R}$ & $U_{\nu}^{(K)}(\frac12,x)$\\
$(0,3)$ & $\nu+\frac12$ & $\sqrt{\nu^2-\frac14}$ & $1$ & $\nu >1/2$ & $U_{\nu}^{(K)}(0,x)$\\
$(1,2)$ & $\nu+\frac12$ & $\nu-\frac12$ & $1$ & $\nu>1/2$ & $L_{\nu}^{(K)}(\frac12,x)$\\
$(3,0)$ & $\nu +2$ & $\nu- 2$ & $\sqrt{(\nu-2)/(\nu-1)}$ & $\nu \ge 2$ & $L_{\nu}^{(K)}(2,x)$\\
& & & & &\\
\end{tabular}

\caption{Bounds for the ratios $I_{\nu-1}(x)/I_{\nu}(x)$ 
and $K_{\nu+1}(x)/K_{\nu}(x)$ 
of the type $B(\alpha,\beta,\gamma,x)=
(\alpha+\sqrt{\beta^2+\gamma^2 x^2})/x$ classified
according to their accuracies at $x=0$ ($n$) and $x=+\infty$ ($m$). The range of validity of the bounds is shown, and the relation with 
the parametric bounds of Theorems \ref{horn}--\ref{blk} is given in the last column (the lower
bounds are denoted with an $L$ and the upper bounds with a $U$).}
\end{table}

The bounds in Table 1 exhaust the best bounds at $0$ and/or $+\infty$ of the form $B(\alpha,\beta,\gamma,x)$. 
However, of course, other forms may be available with higher accuracy. For instance, in \cite{Segura:2023:SBW} the following
bounds are proved using arguments similar to those of Theorem \ref{prime}.
\begin{theorem} 
\label{gapk}
Let $\phi_{-,\nu}(x)=xI_{\nu-1}(x)/I_{\nu}(x)$ and $\phi_{+,\nu}(x)=xK_{\nu+1}(x)/K_{\nu}(x)$,
then both functions satisfy the following properties for $\nu\ge 1/2$ and $x>0$
$$
0<\phi'_{\pm,\nu}(x)\le 1,
$$
$$
B_{\nu}^{(1,3)}(x)\equiv \nu+\sqrt{\nu^2+x(x-1)}<\phi_{\pm,\nu}(x)\le\nu+\sqrt{\nu^2 + x(x+1)}
\equiv \hat{B}_{\nu}^{(1,3)}(x),
$$
where the equality only takes place for 
$\phi_{+,\nu}(x)$ when $\nu=1/2$.
The upper bound for $\phi_{+,\nu}(x)$ and
the lower bound for $\phi_{-,\nu}(x)$ are of accuracy $(1,3)$.
\end{theorem}

Of course, further bounds are possible by application of the recurrence relation. For instance, using the backward recurrence in
the case of first kind Bessel functions we get, starting from the lower bound of the previous theorem that, for all $x>0$ and $\nu> 0$
$$
\Frac{I_{\nu-1}(x)}{I_{\nu}(x)}<\Frac{2\nu}{x}+\Frac{x}{\nu+1+\sqrt{(\nu+1)^2+x(x-1)}},
$$
and this bound has accuracy $(2,3)$. 

\subsection{Other bounds}

It is possible to obtain bounds with higher accuracy using modified methods, as we will next describe. It is also possible,
 as described before, to improve the accuracy of the bounds at $x=0$ by using the recurrence relation. In all these cases, 
 the improvement in the accuracy of the bounds is accompanied by more involved expressions for them. We briefly describe
 the bounds obtained from the iteration of the Riccati equation in 
 \cite{Ruiz:2016:ANT} and from the use of the ODE satisfied by double ratios in \cite{Segura:2021:MPF}.

\subsubsection{Bounds from the iteration of the Riccati equation}

The idea is to start from a Riccati equation
$$
h_0 '(x)=A_0 (x)+B_0(x)h_0 (x)+C_0(x)h_0 (x)^2,
$$
and to consider the function $h_1 (x)=h_0 (x)/\phi_0 (x)$, where $\phi_0 (x)\equiv \beta_0 (x)$ 
is a function of convenience. We choose $\phi_0 (x)$ as one of the roots of $A_0 (x)+B_0(x)\phi_0 (x)+C_0(x)\phi_0 (x)^2=0$, 
which is a bound for $h_0(x)$ if the conditions of Theorem \ref{second} are met. 
The next step
is to consider the Riccati equation for $h_1(x)$
\begin{equation}
\begin{array}{l}
h_1 '(x)=A_1(x)+B_1(x)h_1 (x)+C_1(x) h_1(x)^2\\
A_1 (x)=\Frac{A_0(x)}{\phi_0 (x)},\,B_1 (x) =B_0(x)-\Frac{\phi_0^{\prime}(x)}{\phi_0(x)},\,C_1 (x)=\phi_0(x) C_0(x).
\end{array}
\end{equation}
If $\phi_1 (x)$, one of the solutions of the characteristic equation $A_1 (x)+B_1 (x)\phi_1 (x)+C_1 (x)\phi_1 (x)^2=0$,
turns out to be a bound for $h_1(x)$, then $\beta_1 (x)=\phi_1 (x)\beta_0 (x)=\phi_1(x)\phi_0(x)$ will be a bound for
$h_1 (x)$. For studying whether $h_1(x)$ is a bound, the same type of analysis
as for the original Riccati equation  
is considered, based on Theorem \ref{second}.

This iteration of Riccati equations was introduced in \cite{Ruiz:2016:ANT}, 
starting from the Riccati equations for 
$h(x)=x^{-\alpha}I_{\nu}(x)/I_{\nu-1}(x)$ and $h(x)=x^{-\alpha}K_{\nu -1}(x)/K_{\nu}(x)$ (it would be equivalent to consider the Riccati
equations for $\Phi(x)=1/h(x)$). We summarize these results and compare them with those more elementary (but also accurate) bounds in
\cite{Segura:2023:SBW}.

For first kind Bessel functions, and after one iteration, the best bounds that are obtained are for $\alpha=0$ and $\alpha=2$, and we have:
\begin{theorem}
\label{newbi}
Let
$$
B_{\alpha} (\nu,x)=\Frac{\delta_{\alpha} (\nu,x)+
\sqrt{\delta_{\alpha} (\nu ,x)^2+x^2}}{x},$$
where
$$
\delta_{\alpha} (\nu,x)=(\nu-1/2)+\Frac{\lambda}{2\sqrt{\lambda^2+x^2}},\,\lambda=\nu+(\alpha-1)/2,
$$
then
$$
\Frac{I_{\nu-1}(x)}{I_{\nu}(x)}>B_{0}(\nu,x),\,\nu\ge 1/2
$$
and
$$
\Frac{I_{\nu-1}(x)}{I_{\nu}(x)}<B_{2}(\nu,x),\,\nu\ge 0.
$$
\end{theorem}
The accuracy of the first bound is $(1,3)$ while the second bound has accuracy $(1,2)$. We may compare these bounds with 
the bounds of equal
accuracy obtained in \cite{Segura:2023:SBW}, particularly 
with the bound in Theorem \ref{gapk} for the $(1,3)$ case, which is also a lower bound
of this same accuracy, and with the third bound in Table 1. In both cases, comparing the expansions at $x=0$ and $x=+\infty$ we
conclude that the bounds of the previous theorem are generally better (though more complicated). Numerical tests show that indeed, 
the $(1,2)$ bound of the previous theorem
appears to be better for all $x$ and $\nu>0$ than the bound in Table 1, and that the $(1,3)$ bound of the previous theorem is
also superior when $\nu>3/2$.

The bounds from the previous theorem can be improved 
by applying the recurrence relation in the backward direction. We refer to 
\cite{Ruiz:2016:ANT} for the explicit result. The bounds improve their accuracy at $x=0$ by one unit with respect to the bounds
in Theorem \ref{newbi}, and therefore they have accuracies $(2,3)$ and $(2,2)$. We also refer to \cite{Ruiz:2016:ANT}, Theorem 9, for bounds of the same type for the second kind Bessel function.

\subsubsection{Trigonometric bounds from double ratios}

Similarly as described in Section \ref{beric}, the analysis of the first order ODE satisfied by double ratios of modified Bessel
functions can be used to obtain very sharp trigonometric bounds for ratios of modified Bessel functions, as was described
in \cite{Segura:2021:MPF}. 

We start from the
ratio, $\Phi_{\nu}(x)={\mathcal I}_{\nu-1}(x)/{\mathcal I}_{\nu}(x)$, where ${\mathcal I}_{\nu}(x)$ can be any of the solutions of
(\ref{DDE}). Then, using (\ref{DDE}) one can easily prove that the double ratio 
$$
W_{\nu}(x)=\Frac{\Phi_{\nu}(x)}{\Phi_{\nu+1}(x)}
$$
satisfies the first order ODE
$$
W_{\nu}'(x)=-\Frac{2}{x^3}\left(\psi_\nu (x)^3+\psi_\nu (x)^2-(\nu^2+x^2)\psi_\nu (x)-\nu^2\right)
$$
where
$$W_{\nu}(x)=\Frac{1}{x^2}\left(\psi_\nu (x)^2- \nu^2\right)$$
and we also have
$$
\psi_{\nu}(x)=x\Phi_{\nu}(x)-\nu=x\Frac{{\mathcal I}_{\nu}'(x)}{{\mathcal I}_{\nu}(x)}.
$$
We refer to \cite{Segura:2021:MPF} for further details.

A qualitative analysis of the solution of the ODE for $W_{\nu}(x)$ together with the behavior of the
solutions as $x\rightarrow 0$ and $\rightarrow +\infty$ provides both information on the monotonicity of the
ratios as well as trigonometric bounds for these simple or double ratios. 
In particular, for the first and second
kind Bessel functions, the following result was proved in \cite{Segura:2021:MPF}: 

\begin{theorem}
\label{cotastrig}
For $x>0$ and $\nu\ge 0$ the following holds:
\begin{equation}
\begin{array}{l}
\Frac{I_{\nu-1}(x)}{I_\nu (x)}<\Frac{2g_\nu (x)}{3x}\cos\left(\Frac{1}{3}\arccos\left(
\Frac{h_{\nu}(x)}{g_\nu (x)^3}
\right)\right)+\Frac{\nu-1/3}{x},\\
\Frac{K_{\nu-1}(x)}{K_\nu (x)}<\Frac{2g_\nu (x)}{3x}\cos\left(\Frac{1}{3}\arccos\left(

\Frac{h_{\nu}(x)}{g_\nu (x)^3}
\right)-\Frac{\pi}{3}\right)-\Frac{\nu-1/3}{x},
\end{array}
\end{equation}
where $g_\nu (x)=\sqrt{3(\nu^2+x^2)+1}$, $h_{\nu}(x)=9\nu^2-\Frac{9}{2}x^2-1$.
\end{theorem}

The bound for the first kind function has accuracy $(3,2)$, and the accuracy for the second kind function is
$(2,2)$ when $\nu>1$. It is again possible to use the recurrence to improve the accuracy at $x=0$.

We notice that from the bounds for ratios of modified Bessel functions, it is possible to derive other types of bounds,
like for instance bounds for products of Bessel functions or bounds on Tur\'anians. With respect to the product, we 
mention that it is easy to prove that (see, for example, \cite[section 2.1]{Segura:2021:MPF})
$$
\Frac{I_{\nu-1}(x)}{I_{\nu}(x)}+\Frac{K_{\nu-1}(x)}{K_{\nu}(x)}=\Frac{1}{xI_{\nu}(x)K_{\nu}(x)},
$$
and then, as a consequence, 
\begin{corollary}
$$
I_{\nu}(x)K_{\nu}(x)>\Frac{\sqrt{3}}{2g_\nu (x)
\sin\left(\frac13\arccos\left(
\Frac{h_\nu (x)}{g_\nu (x)^3}\right)+\Frac{\pi}{3}\right)}>\Frac{1}{2\sqrt{x^2+\nu^2+\frac13}}
$$
where $g_\nu (x)=\sqrt{3(\nu^2+x^2)+1}$, $h_{\nu}(x)=9\nu^2-\Frac{9}{2}x^2-1$.
\end{corollary}
We notice that in \cite{Segura:2021:MPF} it was conjectured that
$$I_{\nu}(x)K_{\nu}(x)>\Frac{1}{2\sqrt {x^2+\nu^2+\frac15}}.$$

We don't discuss possible applications of these or the other bounds for bounding Tur\'anians. Bounds for Tur\'anians
are given, for example, in \cite{Segura:2011:BFR,Baricz:2015:BFT}.

\section{Confluent hypergeometric functions}

Modified Bessel functions and parabolic functions are particular cases of confluent hypergeometric functions.
It is therefore natural to analyze whether similar techniques are applicable to more general cases, depending on 
more than one parameter. We start with confluent hypergeometric functions, and in the last section we consider
the Gauss hypergeometric case. The goal is not to be exhaustive with the analysis, as done in the previous examples,
but to illustrate that similar techniques are also fruitful in more general cases.

For confluent hypergeometric functions, a first example of application of the technique based on Riccati equation 
is given in the Appendix of \cite{Segura:2016:SBF}; similar ideas were later considered in \cite{Sablica:2022:OBF}.
Related results can also be obtained by an analysis of the log-concavity and log-convexity of series, as described in \cite{Kalmykov:2013:LCA,Kalmykov:2013:LCF,Sra:2013:TMW}. 

Confluent hypergeometric functions are the solutions of the ODE
\begin{equation}
\label{ODE}
xy''(x)+(b-x)y'(x)-ay(x)=0.
\end{equation}
We consider the regular solution at the origin, that is, the Kummer confluent hypergeometric function
\begin{equation}
\label{series}
M(a,b,x)=\displaystyle\sum_{k=0}^{\infty} \Frac{(a)_n}{(b)_n n!}x^n.
\end{equation}

In our analysis, we prefer to consider an alternative normalization:
$$
m(a,b,x)=\Frac{\Gamma (a)}{\Gamma (b)}M(a,b,x).
$$

With this, we have that $m(a,b,x)$ satisfies the 
difference-differential relation
\begin{equation}
m'(a,b,x)=m(a+1,b+1,x),
\end{equation}
which, considering the differential equation (\ref{ODE}), leads to the recurrence relation
\begin{equation}
\label{TTRR}
xm(a+2,b+2,x)+(b-x)m(a+1,b+1,x)-am(a,b,x)=0.
\end{equation} 

Denoting
$$
h(a,b,x)=\Frac{m'(a,b,x)}{m(a,b,x)}=\Frac{m(a+1,b+1,x)}{m(a,b,x)},
$$
taking the derivative and using the differential equation (\ref{ODE}) we get
\begin{equation}
\label{Ricat}
h'(a,b,x)=\Frac{m''(a,b,x)}{m(a,b,x)}-h(a,b,x)^2=
\Frac{a}{x}+\left(1-\Frac{b}{x}\right)h(a,b,x)-h(a,b,x)^2.
\end{equation}

The characteristic roots of this Riccati equations, $\lambda$--solutions of 
$x\lambda ^2+(b-x)\lambda-a=0$, are
$$
\lambda_{\pm}(a,b,x)=\Frac{x-b\pm \sqrt{(x-b)^2+4ax}}{2x}.
$$
The relevant root will be the positive one, which from now on we denote by $\lambda (a,b,x)$. This root
is increasing if $b>a$, decreasing if $a>b$ and constant if $a=b$. In addition, we see that as
$x\rightarrow 0$
\begin{equation}
\label{la+0}
\lambda (a,b,x)=\Frac{a}{b}\left[1+\Frac{b-a}{b^2}x+{\mathcal O}(x^2)\right].
\end{equation}

On the other hand we have that
\begin{equation}
\label{hm0}
h (a,b,x)=\Frac{m(a+1,b+1,x)}{m(a,b,x)}=\Frac{a}{b}\left[1+\Frac{b-a}{b(b+1)}x+
\Frac{(b-a)(b-2a)}{b^2 (b+1)(b+2)}x^2+\dots \right],
\end{equation}
and therefore $h(a,b,0^+)>0$ and $h'(a,b,0^+)$ has the sign of $b-a$ (same monotonicity as $\lambda (a,b,x)$
close to $x=0$).

The information on the monotonicity of $\lambda (a,b,x)$ and the sign of $h(a,b,0^+)$ and $h'(a,b,0^+)$ 
 is enough to prove the following result, which is 
a direct consequence of Theorem \ref{prime} and which was described earlier in \cite[Thm. 3]{Segura:2016:SBF}.
\begin{theorem}
\label{bo1}
Let us assume that $a,b>0$. Then, $h(a,b,x)=m(a+1,b+1,x)/m(a,b,x)$ is monotonic as a function of $x>0$, 
$h (a,b,x)$ is strictly increasing if $b>a$, constant if $b=a$ and strictly decreasing if
$b<a$. The following inequalities hold for $x>0$:
\begin{enumerate}
\item{}$h(a,b,x)<\lambda (a,b,x)$ if $b>a$.
\item{}$h(a,b,x)=\lambda (a,b,x)=a/b$ if $b=a$.
\item{}$h(a,b,x)>\lambda (a,b,x)$ if $b<a$.
\end{enumerate}
\end{theorem}

As in previous examples, the recurrence relation can be used to obtain further bounds.
We write the recurrence (\ref{TTRR}) as
$$
h(a,b,x)=\Frac{a}{b-x+xh(a+1,b+1,x)}
$$
which is equivalent to applying the recurrence in the backward direction. 

We denote 
$$
\begin{array}{ll}
\tilde{\lambda} (a,b,x)&=\Frac{a}{b-x+x\lambda (a+1,b+1,x)}\\
\\
&= \Frac{2a}{b-x-1+\sqrt{(x-b-1)^2+4(a+1)x}},
\end{array}
$$
which is positive for $a,b,x>0$.

\begin{theorem}
\label{bo3}
Let $a,b,x>0$, then 
$$\tilde{\lambda} (a,b,x)<h (a,b,x)<\lambda (a,b,x) \mbox{ if }b>a.$$ 
The inequalities are
reversed if $b<a$ and they become equalities if $a=b$.
\end{theorem}

Let us now write Theorem \ref{bo3} in terms of the Kummer function:

\begin{theorem}
\label{Ku}
Let $a,b,x>0$. The Kummer function satisfies the inequalities
$$
b-x+\sqrt{(b-x)^2+4ax}<2b\Frac{M(a,b,x)}{M(a+1,b+1,x)}<b-x-1+\sqrt{(x-b-1)^2+4(a+1)x}
$$
if $b>a$ and the inequalities are reversed if $b<a$. The inequalities turn to equalities if $a=b$.
\end{theorem}

We can, as before, measure the accuracy of the bounds by checking how many terms coincide in
the expansions at $x=0$ and $x=+\infty$

 Starting with the upper bound $\lambda (a,b,x)$, comparing (\ref{la+0}) with (\ref{hm0})
 we see that the first term in the expansions at $x=0$ coincide. On the other hand, as 
 $x\rightarrow +\infty$,
 \begin{equation}
 \label{lambdap}
 \lambda (a,b,x)=1-\Frac{b-a}{x}+\Frac{a(b-a)}{x^2}+{\mathcal O}(x^{-3}),
 \end{equation}
 and the first two terms coincide with the expansion of $h (a,b,x)$,
 which is
 \begin{equation}
\label{hmi}
h (a,b,x) =1-\Frac{b-a}{x}+\Frac{(a-1)(b-a)}{x^2}+{\mathcal O}(x^{-3}).
\end{equation}
  With this, we can say that the accuracy of the upper bound is $(1,2)$.
 
 With respect to the accuracy of the lower bounds
  $\tilde{\lambda} (a,b,x)$ for $h (a,b,x)$, considering that
 $$
 \tilde{\lambda} (a,b,x)=\Frac{a}{b}+\Frac{a(b-a)}{b^2 (b+1)}x+{\mathcal O}(x^2)
 $$
 and comparing with (\ref{hm0})
 we see that the first two terms coincide (the third term is not shown but it
 is different). On the other hand,
 $$
 \tilde{\lambda} (a,b,x)=1+\Frac{a(a-b+1)-b}{ax}+{\mathcal O}(x^{-2})
 $$
 and the first term coincides with (\ref{hmi}). Therefore the accuracy for the lower bound $\tilde{\lambda}(a,b,x)$ is $(2,1)$.

 \subsection{Further bounds}
 
 As happened for the case of modified Bessel functions, the bounds that are obtained by the use of the Riccati equation and the
 application of the recurrence relation are of type $(1,2)$ and $(2,1)$. It is natural to ask if, as in the case of Bessel functions, 
 we can obtain uniparametric bounds which continuously connect the $(1,2)$ 
 to $(3,0)$ cases and the $(2,1)$ to 
 the $(0,3)$. We will not answer this question here, but we advance one result in this
  direction, which is the obtention of a bound of the type $(0,3)$.

 A candidate for such $(0,3)$ bound is 
 $\lambda (a-1,b-1,x)$ because, considering (\ref{lambdap}) and (\ref{hmi}) we indeed 
 observe that the first three terms coincide. Considering an additional term in the expansion, we see that
 as $x\rightarrow +\infty$
 \begin{equation}
 \label{de1}
 h (a,b,x)-\lambda (a-1,b-1,x)=\Frac{(a-1)(b-a)}{x^3}+{\mathcal O}(x^{-4}),
 \end{equation}
 if $a>1$, and as $x\rightarrow 0$
 \begin{equation}
 \label{de2}
 h (a,b,x)-\lambda (a-1,b-1,x)=\Frac{b-a}{(b-1)b}+{\mathcal O}(x)
 \end{equation}
 if $b>1$.
 
 \begin{theorem} 
 Let $x>0$ and $a,b>1$, then 
 $B^{(0,3)}(x)\equiv \lambda (a-1,b-1,x)<h (a,b,x)$ if $a<b$, the inequality is reversed if $a>b$ and becomes an
  equality if $a=b$.
 \end{theorem}
 \begin{proof}
 We exclude the trivial case $a=b$.
 
We define
 $$
 \begin{array}{ll}
 \delta (x)=q(x)-\phi (x),\\
 q(x)=1/B^{(0,3)}(x)=\Frac{b-1-x+\sqrt{(x-b+1)^2+4(a-1)x}}{2(a-1)},\\
 \phi (x)=1/h(a,b,x),
 \end{array}
 $$ 
 where $\phi (x)$ satisfies
$$
\phi' (x)=1-\left(1-\Frac{b}{x}\right)\phi (x)-\Frac{a}{x}\phi (x)^2. 
$$

 Considering (\ref{de1}) and (\ref{de2}) we know that $\delta (0^+)$ has the same sign as $b-a$, which is the sign of 
 $\delta (+\infty)$ only if $a>1$. Therefore, we can only have a bound if $a>1$. Next we prove that $\lambda (a-1,b-1,x)$
 is always a bound in that case provided $b>1$.

  Now, for proving the result we apply Theorem \ref{prime}. Because 
  the sign of $\delta (0^+)$ is the same as
the sign of $b-a$, all that remains to be proved is that $(b-a)\Delta(x)>0$ for $x>0$, 
with 
$$
\Delta (x)=q'(x)-1+\left(1-\Frac{b}{x}\right)q(x)+\Frac{a}{x}q^2(x).
$$

After some algebra, one can verify that, as a function of $a$, $\Delta(x)$ only vanishes
at $a=b$ and that 
$$
\left.\Frac{\partial \Delta}{\partial a}\right|_{a=b}=-\Frac{b-1}{x(b-1+x)}<0 \mbox{ if } b>1.
$$
Therefore $\Delta (x)$ has the same sign as $b-a$, which ends the proof.
  
\end{proof}

\subsection{Bounds for other ratios of contiguous functions}

So far, we have considered bounds for the ratios $m(a+1,b+1,x)/m(a,b,x)$, but we could
also consider other ratios like, for instance, $m(a+1,b,x)/m(a,b,x)$
or $m(a+1,b+2,x)/m(a,b,x)$, this last case being
 related to modified Bessel functions, as we will see. The different cases can be related through 
recurrence relations. 

For example, the relation \cite[13.3.4]{Daalhuis:2010:CHF} can be written
\begin{equation}
\label{re1}
m(a+1,b,x)-am(a,b,x)-x m(a+1,b+1,x)=0,
\end{equation}
and therefore
$$
\Frac{m(a+1,b,x)}{m(a,b,x)}=a+x\Frac{m(a+1,b+1,x)}{m(a,b,x)}.
$$
With this the bounds we have obtained so far translate easily to bounds for the ratios $m(a+1,b,x)/m(a,b,x)$
(related with the results of \cite[Thm. 4]{Segura:2016:SBF}). The accuracy of the bounds
 is maintained because there are no subtractions. This is not always the case, and we will illustrate this 
 with the recurrence
 related to the modified Bessel functions.

We consider now the relation \cite[13.3.4]{Daalhuis:2010:CHF}, which we write in the form
\begin{equation}
\label{re2}
m(a+1,b,x)-(b+x)m(a+1,b+1,x)+x(b-a)m(a+1,b+2,x)=0.
\end{equation}

Combining Eqs. (\ref{re1}) and (\ref{re2}) to eliminate $m(a+1,b,x)$ we get
\begin{equation}
a m(a,b,x)=b m(a+1,b+1,x)+x(a-b)m(a+1,b+2,x),
\end{equation}
from which we have
\begin{equation}
\label{12}
\Frac{m(a+1,b+2,x)}{m(a,b,x)}=\Frac{1}{x(a-b)}\left(a-b\Frac{m(a+1,b+1,x)}{m(a,b,x)}\right).
\end{equation}

Using the bounds of Theorem \ref{bo3} on the right-hand side of this equation we get
bounds for $H(a,b,x)=\Frac{m(a+1,b+2,x)}{m(a,b,x)}$. Unlike the case of Theorem 
\ref{bo3} we don't get that the inequalities are reversed when going from the case $a<b$ 
to the case $a>b$ for the $M$ function (notice the $a-b$ in the denominator of (\ref{12})).

We start with the upper bound for $h(a,b,x)$, $\lambda (a,b,x)$, and write
$$
\lambda (a,b,x)=1+f,\,f=\Frac{\sqrt{(x-b)^2+4ax}-x-b}{2x},
$$
which we can also write, denoting $\delta=a-b$
$$
f=\Frac{\sqrt{(x+b)^2+4\delta x}-x-b}{2x}=\Frac{2\delta}{x+b+\sqrt{(x-b)^2+4ax}}.
$$
With this we get the upper bound for $H(a,b,x)$:
$$
\eta (a,b,x)\equiv \Frac{1}{x(a-b)}\left(a-b\lambda (a,b,x)\right)=\Frac{1}{x}\left(1-\Frac{2b}{x+b+\sqrt{(x-b)^2+4ax}}\right).
$$

Proceeding similarly with the lower bound $\tilde{\lambda}(a,b,x)$:
$$
\tilde{\lambda} (a,b,x)=\Frac{2(b+\delta)}{2b+\Delta},
$$
where
$$
\Delta=\sqrt{(x+b+1)^2+4\delta x}-(x+b+1)=\Frac{4\delta x}
{x+b+1+\sqrt{(x+b+1)^2+4\delta x}}.
$$

Now we can write 
$$\tilde{\lambda} (a,b,x)=1+g,\,g=\delta\Frac{1-p}{b+\delta p},\,p=\Frac{2x}{x+b+1+\sqrt{(x+b+1)^2+4(a-b) x}},$$
and we get the lower bound for $H(a,b,x)$ 
$$
\tilde{\eta} (a,b,x)=\Frac{1}{x(a-b)}\left(a-b\tilde{\lambda} (a,b,x)\right)=\Frac{ap}{x(ap+b(1-p))}.
$$

Then we have the following result

\begin{theorem}
For $a,b,x>0$ the following holds
$$\tilde{\eta} (a,b,x)<\Frac{m(a+1,b+2,x)}{m(a,b,x)}<\eta (a,b,x).$$
\end{theorem}

It is instructive to compare now these bounds with those obtained earlier for modified Bessel functions \cite{Segura:2023:SBW} and 
summarized earlier in this paper.
 
Considering the relation \cite[13.6.9]{Daalhuis:2010:CHF}
$$
M(\nu+1/2,2\nu+1,2z)=\Gamma(1+\nu)e^{z}(z/2)^{-\nu}I_{\nu}(z)
$$
we have 
$$
\Frac{I_{\nu}(z)}{I_{\nu-1}(z)}=2z\Frac{m(a+1,b+2,2z)}{m(a,b,2z)},\,a=\nu-1/2,\,b=2\nu-1,
$$
and a straightforward computation shows that the upper bound $\eta (a,b,x)$ in this case corresponds 
to the bound $(0,2)$ in 
\cite[Table 3.1]{Segura:2023:SBW}.

With respect to the lower bound $\tilde{\eta}$, we obtain a bound of type $(1,1)$ which is not in 
\cite[Table 3.1]{Segura:2023:SBW}, namely
$$
\Frac{I_{\nu-1}(x)}{I_{\nu}(x)}<\Frac{\nu+\sqrt{x^2+\nu^2+x}}{x}.
$$
The bound $(1,1)$ of Table 3.1 of \cite{Segura:2023:SBW} is clearly better, because it does not have the
last sumand inside the square root. 

We observe that we have started with bounds with accuracies $(1,2)$ and $(2,1)$ for bounding $m(a+1,b+1,x)/m(a,b,x)$ and 
we have ended with $(0,2)$ and $(1,1)$ bounds for $m(a+1,b+2,x)/m(a,b,x)$. We conclude that for the particular case of the recurrence (\ref{12}), the relation 
with the case we have studied earlier is not convenient because a cancellation appears which reduces the accuracy at $x=0$. 
For this case, and surely for others, an independent analysis is convenient.

\section{Gauss hypergeometric function}

We finally provide some new bounds for the ratios of Gauss hypergeometric functions, and discuss their relation with the 
bounds we have described so far for the confluent hypergeometric case.

As a previous result on bounds of ratios of Gauss hypergeometric functions, we can mention 
\cite[example 3]{Kalmykov:2014:LCO}, where bounds for the ratio ${}_2 F_{1}(a+1,b;c+1;x)/{}_2 F_{1}(a,b;c;x)$ were 
 established. We expect that the analysis based on the qualitative analysis of the Riccati equations can also be used
 to obtain those results, but we choose as an illustration of the Riccati methods the case of the bounds for the
 ratio  ${}_2 F_{1}(a+1,b+1;c+1;x)/{}_2 F_{1}(a,b;c;x)$ and leave for a later analysis the cases of other ratios of 
 Gauss hypergeometric functions.

We define $y(a,b,c,x)=\Frac{\Gamma (a)\Gamma (b)}{\Gamma (c)}\,_{2}{\rm F}_1 (a,b;c;x)$, which satisfies the Gauss differential
equation

\begin{equation}
\label{bog}
x(1-x)y''(a,b,c,x)+\left[c-(a+b+1)x\right]y'(a,b,c,x)-aby(a,b,c,x)=0,
\end{equation}
and the difference-differential relation
$$
y'(a,b,c,x)=y(a+1,b+1,c+1).
$$
Combining both we have the recurrence relation
$$
x(1-x)y(a+2,b+2,c+2,x)+\left[c-(a+b+1)x\right]y(a+1,b+1,c+1,x)-aby(a,b,c,x)=0.
$$

Consider now the ratio 
$$
h(a,b,c,x)=y(a+1,b+1,c+1,x)/y(a,b,x)=y'(a,b,x)/y(a,b,x),
$$
we can write the recurrence relation as
\begin{equation}
\label{backG}
h(a,b,c,x)=\Frac{ab}{c-(a+b+1)x+x(1-x)h(a+1,b+1,c+1,x)}.
\end{equation}
For brevity, except when needed, we drop the parameters $a,b,c$ from the notation of $h(a,b,c,x)$.

Differentiating and using (\ref{bog}) we have
$$
h'(x)=-\Frac{1}{x(1-x)}\left[ x(1-x)h(x)^2+\left[c-(a+b+1)x\right]h(x)-ab\right].
$$
We consider the positive root of the characteristic equation
$$
\lambda (x)=\Frac{(a+b+1)x-c+\sqrt{((a+b+1)x-c)^2+4abx(1-x)}}{2 x(1-x)}.
$$

As $x\rightarrow 0$ we have
\begin{equation}
\label{hgauss}
h(x)=\Frac{a b}{c}\left(1+\Frac{c(a+b+1)-ab}{c(c+1)}x+{\mathcal O}(x^2)\right).
\end{equation}

On the other hand, for $c>0$,
\begin{equation}
\label{lamg}
\lambda (x)=\Frac{a b}{c}\left(1+\Frac{c(a+b+1)-ab}{c^2}x+{\mathcal O}(x^2)\right).
\end{equation}

With this and the following lemma we will have the basic ingredients for obtaining a first bound
for the ratio $h(x)$.

\begin{lemma}
If $a,b,c>0$ with $c>ab/(a+b+1)$ then $\lambda (x)$ is increasing in $[0,1)$.
\end{lemma}
\begin{proof}
We write
$$
\lambda (x)=\Frac{2ab}{\phi (x)},\,
\phi(x)=c-(a+b+1)x+\sqrt{((a+b+1)x-c)^2+4abx(1-x)}
$$
and we prove that $\phi (x)$ is monotonically decreasing in $[0,1)$ if $a,b,c>0$ with
$c>ab/(a+b+1)$.

We have $\phi' (0)=2\left(\Frac{ab}{c}-(a+b+1)\right)$ and then $\phi'(0)<0$ if $c>ab/(a+b+1)$. In
addition, $\phi(0)=2c$ (recall that $c>0$), while
$$
\phi (1)=\left\{
\begin{array}{l}
0,\,a+b+1-c\ge 0\\
2(c-a-b-1),\,a+b+1-c<0.
\end{array}
\right.
$$
Therefore, $\phi(0)>\phi(1)\ge 0$ if $a,\,b,\,c>0$.

On the other hand
$$
\phi''(x)=\Frac{4ab((a+b+1)c-ab-c^2)}{(((a+b+1)x-c)^2+4abx(1-x))^{3/2}},
$$
and we observe that $\phi''(x)$ does not change sign in $[0,1)$ and it has the same 
sign as $f(c)=(a+b+1)c-ab-c^2$. We observe that the quadratic function $f(c)$ is such that $f(\pm \infty)=-\infty$,
and has a maximum at $c=c_m=(a+b+1)/2>0$ where $f(c_m)=\frac14\left((a-b)^2+2(a+b)+1\right)>0$; we
observe that $f(a+b+1)=-ab<0$ and thefore $f(c)<0$ (and $\phi''(x)<0$ in $[0,1)$) if $c\ge a+b+1$.

If $\phi''(x)<0$ then necessarily $\phi' (x)<0$ in $(0,1)$ because $\phi'(0)<0$ and $\phi'(x)$ would decrease in $(0,1)$. 
This is the situation when $c\ge (a+b+1)$, because we have proved that $\phi''(x)<0$ in that case. If $c<a+b+1$ the same would
be true provided $\phi''(x)<0$.

In the cases $c<a+b+1$ for which 
$\phi''(x)>0$, it still holds that $\phi'(x)<0$ in $[0,1)$ under the hypothesis of the theorem, as we prove now. 
We have $\phi'(1)<0$ and $\phi'(x)$ is increasing in $[0,1)$. 
Then, if there existed $x_0\in (0,1)$ such that $\phi'(x_0)=0$ (and $\phi (x_0)>0$ 
because $\phi (x)$ is positive in $(0,1)$) we would have $\phi'(x)>0$ in $(x_0,1)$ because 
$\phi'(x)$ is increasing; but this implies that $0<\phi(x_0)<\phi(1)$, in contradiction with the fact that
$\phi(1)=0$.

\end{proof}

\begin{theorem}
\label{unaGa}
Let $a,b,c>0$, $c>ab/(a+b+1)$ then $h(x)<\lambda (x)$ for all $x\in (0,1)$ and 
$h(x)$ is monotonically increasing in $(0,1)$.
\end{theorem}
\begin{proof}
From (\ref{hgauss}) we see that $h(0^+)>0$, $h'(0^+)>0$. Considering also the monotonicity of $\lambda (x)$, the 
result is an immediate consequence of Theorem \ref{second}.
\end{proof}

The recurrence relation can be used, similarly as we did before in other cases, to obtain further bounds. 
In particular, applying the recurrence (\ref{backG}) to the bound of Theorem \ref{unaGa}, we obtain an 
additional bound. We give those two bounds in terms of the Gauss hypergeometric function in the next theorem.

\begin{theorem}
\label{Gausst}
Suppose $a,b,c>0$, $c>ab/(a+b+1)$,  and denote $$H(x)=2c\Frac{{}_2 F_{1}(a,b;c;x)}{{}_2 F_{1}(a+1,b+1;c+1;x)},$$
$F(x)=4x(1-x)$ and $d=a+b+1$.
The following bounds hold:
$$
H(x)>c-dx+\sqrt{(dx-c)^2+abF(x)},
$$
$$
H(x)<c-1-(d-2)x+\sqrt{((d+2)x-(c+1))^2+(a+1)(b+1)F(x)}.
$$
The validity of the upper bound can be extended to  $c>\Frac{ab-2}{a+b+3}$.
\end{theorem} 

We observe that in the confluent limit we recover the bounds described before for the ratio of confluent hypergeometric functions.
In the confluent limit we make the replacement $x\rightarrow x/b$ and take the limit $b\rightarrow \infty$. After this, in order
to make the connection with the notation for the confluent case (we did not use the parameter $c$) we rename $c$ as $b$. This means
that in the previous theorem $2c\, {}_2 F_{1}(a,b;c;x)/{}_2 F_{1}(a+1,b+1;c+1;x)$ would be replaced by $2bM(a,b,x)/M(a+1,b+1,x)$ and
in the bounds we must consider the replacements $c\rightarrow b$, $(d+m)x\rightarrow x$, $(b+m)F(x)\rightarrow 4x$,  $m$ being any
 constant value. With this, we recover Theorem \ref{Ku} for the case $b>a$.
 
 The case $a<b$ of Theorem \ref{Ku}, however, appears to be disconnected from Theorem \ref{Gausst}.

 \subsection{Future work}

 There are many possibilities for exploring additional bounds for the confluent and Gauss hypergeometric functions. To begin
 with, and comparing with the most well studied case (modified Bessel functions) there is a number of voids
 that need to be filled in order to have a result similar to that of Table 1, 
 where no gaps in the description of bounds exist and there exist uniparametric sets of bounds connecting the best upper
 and lower bounds. Also,
  it is not known whether other types of bounds, like those obtained from the iteration of the Riccati equation 
  (as shown in the Bessel case) 
  or from other type of differential equations (as in the sections for parabolic cylinder functions and Bessel functions) 
  are feasible.

\bibliographystyle{amsalpha}
\bibliography{cont}

\end{document}